\documentclass[reqno,11pt]{article}

\usepackage[normalem]{ulem}%pour barrer du texte

\usepackage{amsmath,amssymb,color,amsthm}
\usepackage{amsfonts, amscd, epsfig, amsmath, amssymb,enumerate}
\usepackage{graphicx}
\usepackage{graphics}
\usepackage{color}
\usepackage{mathrsfs}
\usepackage{todonotes}
\usepackage{soul}
\usepackage{authblk}

 \usepackage[usenames,dvipsnames]{pstricks}
 \usepackage{pstricks-add}
 \usepackage{epsfig}
 \usepackage{pst-grad} % For gradients
 \usepackage{pst-plot} % For axes
 \usepackage[space]{grffile} % For spaces in paths
 \usepackage{etoolbox} % For spaces in paths
\usepackage[margin=3cm]{geometry}
 \makeatletter % For spaces in paths
 \patchcmd\Gread@eps{\@inputcheck#1 }{\@inputcheck"#1"\relax}{}{}
 \makeatother

\newtheorem{theorem}{Theorem}[section]
\newtheorem{lemma}[theorem]{Lemma}

\newtheorem{corollary}[theorem]{Corollary}
\newtheorem{remark}[theorem]{Remark}

\newtheorem{assumption}{Assumption}[section]

\usepackage{tikz}
\usetikzlibrary{backgrounds}
\usetikzlibrary{patterns,fadings}
\usetikzlibrary{arrows,decorations.pathmorphing}
\usetikzlibrary{decorations}
\usetikzlibrary{calc}
\usetikzlibrary{shapes.misc}

\usepackage{setspace}
\definecolor{light-gray}{gray}{0.95}
\usepackage{float}
\usepackage[colorlinks=true,linkcolor=blue,citecolor=magenta]{hyperref}
\def\centerarc[#1](#2)(#3:#4:#5){\draw[#1] ($(#2)+({#5*cos(#3)},{#5*sin(#3)})$) arc (#3:#4:#5);}

\usepackage{titlesec}%修改层次标题格式宏包
\titleformat{\subsection}[runin]{\it}{\thesubsection}{4 pt}{ \it}[]%设置subsection格式
\titleformat*{\section}{\bf \center}

\numberwithin{equation}{section}
\numberwithin{figure}{section}

\newcommand{\<}{\left\langle}
\renewcommand{\>}{\right\rangle}

\renewcommand{\epsilon}{\varepsilon}

\newcommand{\R}{\mathbb R}
\newcommand{\Z}{\mathbb Z}
\newcommand{\N}{\mathbb N}
\renewcommand{\P}{\mathbb P}
\newcommand{\T}{\mathbb T}
\newcommand{\E}{\mathbb E}

\newcommand{\gen}{\mathscr{L}_{N}}

\newcommand{\Ocal}{\mathcal O}

\newcommand{\CS}{Cauchy-Schwarz inequality}

\allowdisplaybreaks %%Pour \UTF{00E9}viter les grands espacements verticaux.

\title{Equilibrium Perturbations for Asymmetric Zero Range Process under Diffusive Scaling in Dimensions $d \geq 2$}

\date{}

\author{Linjie Zhao}

\begin{document}

\maketitle

\begin{abstract}
We consider the asymmetric zero range process in dimensions $d \geq 2$.  Assume the initial density profile is a perturbation of the constant density, which has order $N^{-\alpha}$, $\alpha \in (0,1)$, and is constant along the drift direction.  Here, $N$ is the scaling parameter. We show that under some constraints on the jump rate of the zero range process, the perturbed quantity macroscopically obeys the heat equation under diffusive scaling.\\
	
\end{abstract}

\section{Introduction}

It is well known that for asymmetric interacting particle systems with only one conservation law, such as asymmetric zero range or exclusion processes, the macroscopic density profile obeys the hyperbolic equation under hyperbolic scaling, \emph{i.e.}  time sped up by $N$ and space divided by $N$ \cite{rezakhanlou91}.  In order to understand Navier-Stokes equations from a microscopic point of view, asymmetric interacting particle systems have also been considered under diffusive scaling, \emph{i.e.}  time sped up by $N^2$ and space divided by $N$. In the seminal paper \cite{esposito1994diffusive}, Esposito, Marra and Yau prove that for asymmetric exclusion processes in dimensions $d \geq 3$, if the initial density profile is a perturbation of order $N^{-1}$ with respect to the constant density, then the perturbed quantity evolves according to a parabolic equation under diffusive scaling.  In the literature this is called the incompressible limit, which  has also been extended to boundary driven asymmetric exclusion process in \cite{benois2002hydrodynamics}. Another interpretation of the Navier-Stokes equation is to describe the evolution of system in the hyperplane orthogonal to the drift. Precisely speaking, in \cite{benois1997diffusive,landim2004hydrodynamic}, it has been proven that for asymmetric zero range and exclusion processes, if the initial density profile is constant along the drift direction, then the macroscopic behavior is described by a parabolic equation under diffusive scaling. The third interpretation is to consider the first order correction to the hydrodynamic equation, which however is under  hyperbolic scaling. We refer to \cite[P. 185]{klscaling} for more background and references on understanding of Navier-Stokes equations.  

In this note, we consider equilibrium perturbations for the asymmetric zero range process under diffusive scaling in dimensions $d \geq 2$.  We first recall the result in \cite{benois1997diffusive}.  Consider the zero range process $(\eta_t, t \geq 0)$ with jump rate $N^2 g(\cdot)$ and transition probability $p(\cdot)$ such that $m:= \sum_{x \in \T_N^d} x p(x) \neq 0$.  Here, $N$ is the scaling parameter. Assume the initial distribution of the process is associated to some  density profile $\varrho_0: \T^d \rightarrow \R_+$ such that $m \cdot \nabla \varrho_0 (u) = 0$ for any $u \in \T^d$. Then, it is proved in \cite{benois1997diffusive} that for any continuous function $F: \T^d \rightarrow \R$,
\[\lim_{N \rightarrow \infty} \frac{1}{N^d} \sum_{x \in \T_N^d} \eta_{t} (x) F (\tfrac{x}{N}) = \int_{\T^d} \varrho (t,u) F(u) du \quad \text{in probability},\]
where $\varrho (t,u)$ is the solution to the following parabolic equation
\begin{equation}\label{parabolicEqn}
	\begin{cases}
	\partial_t \varrho (t,u) = \frac{1}{2} \sum_{i,j=1}^d \sigma_{i,j} \partial^2_{u_i,u_j} \Phi (\varrho (t,u)),\\
	\varrho (0,u) = \varrho_0 (u).
\end{cases}\end{equation}
Above, $\sigma_{i,j} = \sum_{x \in \T_N^d} x_i x_j p(x)$ for $1 \leq i,j \leq d$ and $\Phi (\varrho)$ is the expectation of the jump rate $g(\cdot)$ with respect to the invariant measure of the process with density $\varrho$.

We assume the initial density profile has the following form: fix a positive constant $\rho_*$ and for $\alpha \in (0,1)$, let
\[\varrho_0 (u)= \rho_* + N^{-\alpha} \rho_0 (u).\]
Assume further that $m \cdot \nabla \rho_0 (u) = 0$ for any $u \in \T^d$.  Then, under some constraints on the jump rate $g(\cdot)$, we prove that for any continuous function $F: \T^d \rightarrow \R$,
\[\lim_{N \rightarrow \infty} \frac{1}{N^{d-\alpha}} \sum_{x \in \T_N^d}\big( \eta_{t} (x) - \rho_*\big) F (\tfrac{x}{N}) = \int_{\T^d} \rho (t,u) F(u) du \quad \text{in probability},\]
where $\rho (t,u)$ is the solution to the following heat equation
\begin{equation}\label{hE}
\begin{cases}
	\partial_t \rho (t,u) = \frac{1}{2} \sum_{i,j=1}^d \sigma_{i,j} \Phi^\prime (\rho_*) \partial^2_{u_i,u_j} \rho (t,u),\\
	\rho (0,u) = \rho_0 (u).
\end{cases}
\end{equation}
Note that if we assume a priori that the density profile of the process at time $t$ is given by \[\varrho (t,u) := \rho_* + N^{-\alpha} \rho(t,u)\] and substitute $\varrho (t,u)$ into the parabolic equation \eqref{parabolicEqn}, then we get the above  heat equation by using Taylor's expansion and by letting $N \rightarrow \infty$.

The proof is based on relative entropy method introduced by Yau \cite{yau1991relative}.  Precisely speaking, let $\nu_{N,t}$ be the product measure with macroscopic density profile $\rho_* + N^{-\alpha} \rho (t,u)$, and let $\mu_{N,t}$ be the distribution of the process at time $t$. We prove in Theorem \ref{thm} that per volume of the relative entropy of $\mu_{N,t}$ with respect to $\nu_{N,t}$ is of order $o(N^{-2\alpha})$. From this and entropy inequality, it is easy to obtain law of large numbers for the perturbed quantities. The main step to bound the relative entropy is to prove a quantitative version of the so-called one block estimate, for which we present two different proofs: one uses spectral gap estimate and the other uses logarithmic Sobolev inequality. The proof based on spectral gap estimate follows the steps in \cite{toth2002between}.   Logarithmic Sobolev inequality has also been used to quantify block estimates in the theory of hydrodynamic limit, see \cite{fritz2004derivation} for example.

Equilibrium perturbations from equilibrium have also been considered in other contexts. For example, in dimension one, Sepp{\"a}l{\" a}inen \cite{seppalainen2001perturbation} considers the  Hammersley’s model, adds a perturbation of order $N^{-\alpha}$ to the equilibrium,  and shows that the perturbation macroscopically obeys the invisid Burgers equation in the time scale $N^{1+\alpha}t$ if $0 < \alpha< 1/2$, which  is extended by T{\'o}th and Valk{\'o} \cite{toth2002between} to a large class of one-dimensional interacting particle systems based on relative entropy method, but only for  $0 < \alpha < 1/5$ and only in the smooth regime of the solution.  For systems with two conservation laws,
T{\'o}th and Valk{\'o} \cite{toth2005perturbation} obtain a two-by-two system  for a very rich class of systems. In \cite{valko2006hydrodynamic}, Valk{\'o} shows that small perturbations around a hyperbolic equilibrium point evolve according to two decoupled Burgers equations. For very rencent results, we refer to \cite{jara2021viscous} for equilibrium perturbations in weakly asymmetric exclusion processes, and to \cite{xu2022equilibrium} for the generalized exclusion process and anharmonic chains.

The rest of the note is organized as follows. In Section \ref{sec:results}, we state the model and main results.  The proof of Theorem \ref{thm} is presented in Section \ref{sec:relativeentropy} by assuming  the one block estimate holds, whose proof is postponed to Section \ref{sec:oneblock}.

\section{Notation and Results}\label{sec:results}

The state space of the zero range process is $\Omega_N^d = (\T_N^d)^{\N}$, where $\N = \{0,1,2,\ldots\}$ and $\T_N^d = \Z^d / (N\Z^d)$ is the $d$ dimensional discrete torus with $d \geq 2$.  For a configuration $\eta \in \Omega_N^d$, $\eta(x)$ is the number of particles at site $x$. Let $p (\cdot)$ be a probability measure on $\Z^d$. We assume that  
\begin{enumerate}[(i)]
	\item $p(\cdot)$ is of finite range: there exists $R > 0$ such that $p(x) = 0$ for all $|x| > R$;
	\item  $p(\cdot)$ is asymmetric: $m := \sum_{x \in\T_N^d} x p(x) \neq 0$.
\end{enumerate}
Above, $|x| = \max_{1 \leq i \leq d} |x_i|$ for $x \in \Z^d$. Let $g: \N \rightarrow \R_+$ be the jump rate of the zero range process.  To avoid degeneracy, assume  $g(k) = 0$ if and only if $k=0$.
The generator of the zero range process acting on functions $f: \Omega_N^d \rightarrow \R$ is given by
\[\gen f (\eta) = \sum_{x,y\in \T_N^d} g(\eta(x)) p(y) \big[ f(\eta^{x,x+y}) - f(\eta) \big].\]
Here, $\eta^{x,y}$ is the configuration obtained from $\eta$ after a particle jumps from $x$ to $y$,
\[\eta^{x,y} (z) = \begin{cases}
	\eta(x) - 1, \quad &z=x,\\
	\eta(y) + 1, \quad &z=y,\\
	\eta(z), \quad &z\neq x,y.\\
\end{cases}\]

Throughout the paper, we need the following assumptions on the rate function $g(\cdot)$.

\begin{assumption}\label{assump:g}
(i) There exists a constant $a_0$ such that for any $k \geq 1$, $|g(k+1) - g(k)| \leq a_0$;\\
(ii) there exists $k_0 > 0$ and $a_1 > 0$ such that $g(k) - g(j) > a_1$  for any $k \geq j+k_0$.
\end{assumption}

\begin{remark}
	Condition $(i)$ is needed in order the process to be well defined in the infinite volume (\emph{cf.}\;\cite{Andjel82}).  Condition $(ii)$ ensures spectral gap estimates and logarithmic Sobolev inequality for the zero range process (\emph{cf.}\;\cite{landim1996spectral,pra2005logarithmic}), which are main tools of this note.
\end{remark}

\begin{remark}
We assume condition $(ii)$ for simplicity. The results in this note should also hold as long as the spectral gap of the zero range process shrinks at rate at least $\ell^{-2}$ (not necessarily uniformly in the particle density), where $\ell$ is the size of the underlying lattice.  For example, the  spectral gap estimates have also been proven in the following cases: (a) $g(k) = k^{\alpha}$, $\alpha \in (0,1)$, see \cite{nagahata2010spectral}; (b) $g(k) = \mathbf{1}_{\{k \geq 1\}}$, see \cite{morris2006spectral}.
\end{remark}

It is well known that the zero range process has a family of product invariant measures indexed by the particle density. Precisely speaking, for each $\varphi \geq 0$, let $\bar{\nu}^N_\varphi$ be the product measure on $\Omega_N^d$ with marginals given by
\[\bar{\nu}^N_\varphi (\eta(x) = k) = \frac{1}{Z(\varphi) } \frac{\varphi^k}{g(k)!}, \quad k \geq 0, \;x \in \T_N^d.\]
Here, $Z (\varphi)$ is the normalizing constant and $g(k)! = \prod_{j=1}^k g(j)$ with the convention that $g(0)! = 1$.  Under Assumption \ref{assump:g}, $Z(\varphi) < \infty$ for any $\varphi \geq 0$. In particular,
\[E_{\bar{\nu}^N_\varphi} [e^{\lambda \eta(0)}] < \infty, \quad \forall \lambda \geq 0.\]
For $\varphi \geq 0$, the particle density under $\bar{\nu}^N_\varphi$ is 
\[R(\varphi) = E_{\bar{\nu}^N_\varphi} [\eta(x)]. \]
It is easy to see that $R(\varphi)$ is strictly increasing in $\varphi$, hence has an inverse denoted by $\Phi := R^{-1}$. To index the invariant measures by particle density $\rho \geq 0$, denote $\nu^N_\rho := \bar{\nu}^N_{\Phi (\rho)}$.

\medspace

Fix $\rho_* > 0$. Let $\rho_0 : \T^d \rightarrow \R_+$ be the initial density profile of the perturbed  quantity. We assume $\rho_0$ is continuously differentiable  and is constant along the drift direction
\begin{equation}\label{assump:initial}
m \cdot \nabla \rho_0 (u) = 0, \quad \forall u \in \T^d.
\end{equation}
Let $\alpha > 0$ denote the strength of the perturbed quantity. The initial distribution of the process is
\begin{equation}\label{initialdistri}
\mu_{N,0} (d \eta)=  \bigotimes_{x \in \T_N^d} \nu^1_{\rho_* + N^{-\alpha} \rho_0 (x/N)} (d \eta(x)).
\end{equation}
Denote by $\mu_{N,t}$ the distribution of the process with generator $N^2 \gen$ at time $t$ starting from $\mu_{N,0}$. The corresponding accelerated process is denoted by $(\eta_t,t \geq 0)$.

Let $\rho(t,u)$ be the solution to the following heat equation
\begin{equation}\label{heatEqn}
	\begin{cases}
	\partial_t \rho (t,u) = \frac{1}{2}\sum_{i,j = 1}^d \Phi^\prime (\rho_*) \sigma_{i,j} \partial^2_{u_i,u_j} \rho (t,u),\\
	\rho(0,u) = \rho_0 (u).
\end{cases}\end{equation}
Above, $\sigma_{i,j} = \sum_{x \in\T_N^d}x_i x_j p(x)$ for $ 1 \leq i,j \leq d$.  Define the reference measure as 
\[\nu_{N,t} (d \eta) = \bigotimes_{x \in \T_N^d} \nu_{\rho_*+N^{-\alpha} \rho (t,x/N)}^1 (d \eta (x)). \]

For two probability measures $\mu,\nu$ on $\Omega_N^d$ such that $\mu$ is absolutely continuous with respect to $\nu$, recall the relative entropy of $\mu$ with respect to $\nu$ is defined as 
\[H(\mu|\nu) = \int \log (d \mu / d \nu) d \mu.\]
To make notation simple, denote
\[h_N(t) = N^{-d} H(\mu_{N,t} | \nu_{N,t}).\]

The following is the main result of this note.

\begin{theorem}\label{thm} Let $d \geq 2$. For any $\alpha \in (0,1)$, we have $h_N (t) = o(N^{-2 \alpha})$.
\end{theorem}

\begin{remark}
	We do not actually need the initial measure $\mu_{N,0}$ to be product as given in \eqref{initialdistri}. But we need it to satisfy: $(i)\;  h_N(0) = o(N^{-2\alpha})$, $(ii)\; N^{-d} H(\mu_{N,0} | \nu^N_{\rho_*}) \leq C N^{-2\alpha}$ for some constant $C$ independent of $N$. 
\end{remark}

For any probability measure $\mu$ on $\Omega_N^d$, denote by $\P_\mu$ the distribution of the process $(\eta_t,t\geq 0)$ with initial distribution $\mu$, and by $\E_\mu$ the corresponding expectation. As a direct consequence of the above theorem, we have the following law of large numbers, whose proof uses entropy inequality and is very standard (\emph{cf.}\;\cite[Corollary 6.1.3]{klscaling} for example). For this reason, we omit the proof.

\begin{corollary}
Under the assumptions of Theorem \ref{thm}, for any continuous function $F: \T^d \rightarrow \R$, for any $t>0$ and for any $\varepsilon > 0$,
\[\lim_{N \rightarrow \infty} \P_{\mu_{N,0}} \Big( \Big| \frac{1}{N^{d-\alpha}} \sum_{x\in \T_N^d} \big( \eta_t (x) - \rho_*\big) F \big(\tfrac{x}{N}\big) - \int_{\T^d} \rho(t,u) F(u) du\Big| > \varepsilon \Big) = 0.\]
Above, $\rho(t,u)$ is the solution to the heat equation \eqref{heatEqn}.
\end{corollary}

\section{Relative Entropy}\label{sec:relativeentropy}

In this section, we first bound the entropy production of the process in Subsection \ref{subsec:entropy}, then calculate the relative entropy in Subsection \ref{subsec:calculations}, and finally prove Theorem \ref{thm} in the last subsection.  To make notations short, we denote $\rho_N (t,u) := \rho_* + N^{-\alpha} \rho (t,u)$. Recall $\rho(t,u)$ is the solution to the heat equation \eqref{heatEqn}.  We also underline that in the following proof the constant $C$ may be different from line to line, but does not depend on the scaling parameter $N$.

\subsection{Entropy production.}\label{subsec:entropy}  For $s \geq 0$, denote
\[f_s = f_{N,s} = \frac{d \mu_{N,s}}{d \nu_{\rho_*}^N}.\]
For any $\nu_{\rho_*}^N$--density $f$, the Dirichlet form of $f$ with respect to $\nu_{\rho_*}^N$ is defined as
\[D_N (f;\nu_{\rho_*}^N) := \<\sqrt{f} (-\gen)\sqrt{f}\>_{\nu_{\rho_*}^N}. \]
Here, for any function $f$ and any distribution $\mu$ on $\Omega_N^d$, $\langle f \rangle_\mu := \int f du$. Since $\nu_{\rho_*}^N$ is invariant for the generator $\gen$, direct calculations show that
\[D_N (f;\nu_{\rho_*}^N) = \sum_{x,y \in \T_N^d} D_{x,y} (f;\nu_{\rho_*}^N),\]
where $D_{x,y}$ is the piece of Dirichlet form associated to the bond $(x,y)$,
\[D_{x,y} (f;\nu_{\rho_*}^N) = \frac{1}{2} \int g(\eta(x)) s(y-x) \big[ \sqrt{f(\eta^{x,y})} - \sqrt{f(\eta)} \big]^2 d \nu_{\rho_*}^N.\]
Above, $s(x) = [p(x)+p(-x)]/2$.

\begin{lemma}\label{lem:entropy}
There exists a constant $C > 0$ independent of $N$ such that for any $t > 0$,
\[H(\mu_{N,t}| \nu_{\rho_*}^N)  + N^2 \int_0^t D_N (f_s;\nu_{\rho_*}^N) ds \leq C N^{d-2\alpha}.  \]
\end{lemma}

\begin{proof}
Since $\nu_{\rho_*}^N$ is invariant for the zero range process, following the proof in \cite[Subsection 5.2]{klscaling} line by line, we have 
\[H(\mu_{N,t}| \nu_{\rho_*}^N)  + N^2 \int_0^t D_N (f_s;\nu_{\rho_*}^N) ds \leq  H(\mu_{N,0}| \nu_{\rho_*}^N). \]
Therefore, to conclude the proof, we only need to show
\[H(\mu_{N,0}| \nu_{\rho_*}^N) \leq C N^{d-2\alpha}.\]
By direct calculations,
\begin{multline*}
H(\mu_{N,0}| \nu_{\rho_*}^N)  = \int \log \frac{\mu_{N,0} (\eta)}{\nu_{\rho_*}^N (\eta)} \mu_{N,0} (d\eta) \\
= \sum_{x \in \T_N^d} \Big\{ \log \frac{Z(\Phi (\rho_*))}{Z(\Phi ( \rho_N (0,\tfrac{x}{N}))} + \rho_N (0,\tfrac{x}{N}) \log \frac{\Phi (\rho_N (0,\tfrac{x}{N}))}{\Phi(\rho_*)} \Big\}.
\end{multline*}
Using the basic inequality $\log (1+x) \leq x$, we bound the last line by
\begin{equation}\label{d1}
	 \sum_{x \in \T_N^d} \Big\{  \frac{Z(\Phi (\rho_*) )- Z(\Phi (\rho_N (0,\tfrac{x}{N}))}{Z(\Phi (\rho_N (0,\tfrac{x}{N}))} 
	 + \rho_N (0,\tfrac{x}{N}) \frac{\Phi (\rho_N (0,\tfrac{x}{N})) - \Phi (\rho_*)}{\Phi(\rho_*)} \Big\}.
\end{equation}
Note that  for any $\rho > 0$, 
\[\frac{Z^\prime (\Phi (\rho))}{Z(\Phi (\rho))} = \frac{\rho}{\Phi (\rho)}.\]
Then, by Taylor's expansion, the first term inside the brace in \eqref{d1} equals
\[\frac{\rho_N (0,\tfrac{x}{N})}{\Phi (\rho_N (0,\tfrac{x}{N}))} \big[ \Phi (\rho_* ) - \Phi (\rho_N (0,\tfrac{x}{N})) \big]+ \Ocal (N^{-2\alpha}).\]
Therefore, we may bound the term \eqref{d1} by
\[
	\sum_{x \in \T^d_N} \rho_N (0,\tfrac{x}{N})  \big(\Phi (\rho_N (0,\tfrac{x}{N})) - \Phi (\rho_*) \big) \Big[\frac{1}{\Phi (\rho_*)} - \frac{1}{\Phi (\rho_N (0,\tfrac{x}{N}))}\Big] + \Ocal (N^{d-2\alpha}).
\]
We conclude the proof by noting that $|\Phi (\rho_N (0,\tfrac{x}{N})) - \Phi (\rho_*)| \leq C N^{-\alpha}$.
\end{proof}

\subsection{Calculations.}\label{subsec:calculations}  Let
\[\psi_{N,t} := \frac{d \nu_{N,t}}{d \nu_{\rho_*}^N}.\]
Since $\nu_{N,t}$ and $\nu_{\rho_*}^N$ are both product measures, $\psi_{N,t}$ is explicitly given by
\[\psi_{N,t} (\eta)= \prod_{y \in \T_N^d} \frac{Z(\Phi(\rho_*))}{Z(\Phi (\rho_N(t,\tfrac{y}{N})))}  \Big[\frac{\Phi (\rho_N(t,\tfrac{y}{N}))}{\Phi (\rho_*)}\Big]^{\eta(y)}.\]
Using Yau's relative entropy inequality (\emph{cf.}\;\cite[Lemma 6.1.4]{klscaling} for example),
\begin{equation}\label{entropy0}
\frac{dh_N(t)}{dt} \leq N^{-d} \int \big\{ \psi_{N,t}^{-1} N^{2} \gen^* \psi_{N,t} - \partial_t \log \psi_{N,t} \big\} d \mu_{N,t}.
\end{equation}
Above, $\gen^*$ is the adjoint generator of $\gen$ in $L^2 (\nu_{\rho_*}^N)$ and corresponds to the zero range process with jump rate $p(-\cdot)$.  Precisely speaking, for any $f: \Omega_N^d \rightarrow \R$,
\[\gen^* f (\eta) = \sum_{x,y\in \T_N^d} g(\eta(x)) p(-y) \big[ f(\eta^{x,x+y}) - f(\eta) \big].\]

\medspace

Now we calculate the right hand side of \eqref{entropy0}. By direct calculations,
\begin{multline}\label{lPsi}
	\psi_{N,t}^{-1} N^{2-d} \gen^* \psi_{N,t} 
	= N^{2-d} \sum_{x,y\in \T_N^d} g(\eta(y)) p(y-x)  \bigg[ \frac{\Phi (\rho_N(t,\tfrac{x}{N}))}{\Phi (\rho_N(t,\tfrac{y}{N}))} - 1 \bigg]\\
	= - N^{1-d} \sum_{y\in\T_N^d} \sum_{i=1}^d   \frac{m_i \partial_{u_i} \Phi (\rho_N(t,\tfrac{y}{N}))}{\Phi (\rho_N(t,\tfrac{y}{N}))} \big[g(\eta(y)) - \Phi (\rho_N(t,\tfrac{y}{N}))\big]\\
	+ \frac{1}{2N^d}  \sum_{y\in\T_N^d} \sum_{i,j=1}^d  \frac{\sigma_{i,j} \partial^2_{u_i,u_j} \Phi (\rho_N(t,\tfrac{y}{N}))}{\Phi (\rho_N(t,\tfrac{y}{N}))} \big[g(\eta(y)) - \Phi (\rho_N(t,\tfrac{y}{N}))\big] + \mathcal{E}_{N,1} + \Ocal (N^{-1-\alpha}).
\end{multline}
Above, the term $\Ocal (N^{-1-\alpha})$ comes from the errors by Taylor's expansion since there exists some constant $C$ independent of $N$ such that
\[\sup_{u \in \T^d} \big| \partial^3_{u_i,u_j,u_k} \Phi (\rho_N(t,u)) \big| \leq C N^{-\alpha}, \quad \forall 1 \leq i,j,k \leq d,\]
and the other error term is given by
\begin{equation}\label{en1}
	\mathcal{E}_{N,1} = - N^{1-d} \sum_{y\in\T_N^d} \sum_{i=1}^d m_i \partial_{u_i} \Phi (\rho_N(t,\tfrac{y}{N}))
	+ \frac{1}{2N^d} \sum_{y\in\T_N^d} \sum_{i,j=1}^d \sigma_{i,j} \partial^2_{u_i,u_j} \Phi (\rho_N(t,\tfrac{y}{N})).
\end{equation}

\begin{lemma}\label{lem:error1}
There exists a constant $C$ independent of $N$ such that
	\[|\mathcal{E}_{N,1} | \leq C \big(N^{-1-\alpha} + N^{-3\alpha}\big).\]
In particular, since $\alpha \in (0,1)$, $\mathcal{E}_{N,1} = o(N^{-2\alpha})$.
\end{lemma}

\begin{proof} By direct calculations, for $1 \leq i,j \leq d$,
	\begin{align}
	\partial_{u_i} \Phi (\rho_N(t,u)) &= N^{-\alpha} \Phi^\prime (\rho_N(t,u)) \partial_{u_i} \rho (t,u),\label{firstDer}\\
	\partial^2_{u_i,u_j} \Phi (\rho_N(t,u)) &= N^{-\alpha} \Phi^\prime (\rho_N(t,u)) \partial^2_{u_i,u_j} \rho (t,u) 
+ N^{-2\alpha} \Phi^{\prime \prime}(\rho_N(t,u)) \partial_{u_i} \rho (t,u)  \partial_{u_j} \rho (t,u).\label{secondDer}
	\end{align}
By assumption \eqref{assump:initial}, 
\begin{equation}\label{nablaZero}
m \cdot \nabla \rho(t,u) = 0
\end{equation}
for all $t \geq 0$ and for all $u \in \T^d$. Therefore,  the first term on the right hand side of \eqref{en1} is equal to zero. Since
\[\Phi^\prime (\rho_N(t,u)) = \Phi^\prime (\rho_*) + N^{-\alpha} \Phi^{\prime \prime} (\rho_*) \rho(t,u) + \Ocal (N^{-2\alpha}), \quad \Phi^{\prime \prime} (\rho_N(t,u)) = \Phi^{\prime \prime}  (\rho_*) + \Ocal (N^{-\alpha}),\]
by \eqref{secondDer},  the second term in \eqref{en1} equals
\begin{multline}\label{en1-1}
	\frac{1}{2N^{d+\alpha}} \sum_{y\in\T_N^d} \sum_{i,j = 1}^d \Phi^\prime (\rho_*) \sigma_{i,j} \partial_{u_i,u_j}^2 \rho(t,\tfrac{y}{N}) \\
	+ \frac{1}{2N^{d+2 \alpha}} \sum_{y\in\T_N^d} \sum_{i,j = 1}^d \Phi^{\prime\prime} (\rho_*)  \sigma_{i,j} \big[ \rho(t,\tfrac{y}{N}) \partial_{u_i,u_j}^2 \rho(t,\tfrac{y}{N}) + \partial_{u_i} \rho(t,\tfrac{y}{N}) \partial_{u_j} \rho(t,\tfrac{y}{N}) \big] \\
	+ \Ocal (N^{-3\alpha}).
\end{multline}
Since 
\[\Big| \frac{1}{N^d} \sum_{y\in\T_N^d} \sum_{i,j = 1}^d \Phi^\prime (\rho_*) \sigma_{i,j} \partial_{u_i,u_j}^2 \rho(t,\tfrac{y}{N}) - \int_{\T^d} \sum_{i,j=1}^d \Phi^\prime (\rho_*) \sigma_{i,j} \partial_{u_i,u_j}^2 \rho(t,u) du \Big| \leq \frac{C}{N},\]
and by the conservation of particle numbers,
\[\int_{\T^d} \frac{1}{2} \sum_{i,j=1}^d \Phi^\prime (\rho_*) \sigma_{i,j} \partial_{u_i,u_j}^2 \rho(t,u) du = \partial_t \int_{\T^d} \rho(t,u) du =0, \]
the first term in \eqref{en1-1} is bounded by $C N^{-1-\alpha}$. Similarly, the second term in \eqref{en1-1} is bounded by $C N^{-1-2\alpha}$ since
\[\int_{\T^d} \sum_{i,j=1}^d \sigma_{i,j} \big[ \rho(t,u)  \partial_{u_i,u_j}^2 \rho(t,u)  + \partial_{u_i} \rho(t,u) \partial_{u_j} \rho (t,u)\big]du = 0.\]
This concludes the proof.
\end{proof}

By \eqref{firstDer} and \eqref{nablaZero}, the first term on the right hand side of \eqref{lPsi} equals zero. By \eqref{secondDer} and Lemma \ref{lem:error1}, we rewrite  $\psi_{N,t}^{-1} N^{2-d} \gen^* \psi_{N,t}$ as
\begin{multline}\label{lPsi1}
	\psi_{N,t}^{-1} N^{2-d} \gen^* \psi_{N,t}  \\=  \frac{1}{2N^{d+\alpha}}  \sum_{y\in\T_N^d} \sum_{i,j=1}^d \frac{\sigma_{i,j} \Phi^\prime (\rho_N(t,\tfrac{y}{N})) \partial^2_{u_i,u_j} \rho (t,\tfrac{y}{N})}{\Phi (\rho_N(t,\tfrac{y}{N}))} 
\big[g(\eta(y)) - \Phi (\rho_N(t,\tfrac{y}{N}))\big] \\
	+ \frac{1}{2N^{d+2\alpha}}  \sum_{y\in\T_N^d} \sum_{i,j=1}^d  \frac{ \sigma_{i,j} \Phi^{\prime \prime} (\rho_N(t,\tfrac{y}{N})) \partial_{u_i} \rho (t,\tfrac{y}{N})  \partial_{u_j} \rho (t,\tfrac{y}{N})}{\Phi (\rho_N(t,\tfrac{y}{N}))} \big[g(\eta(y)) - \Phi (\rho_N(t,\tfrac{y}{N}))\big] + o(N^{-2\alpha}).
\end{multline}

Direct calculations yield
\[N^{-d} \partial_t \log \psi_{N,t} = \frac{1}{N^d} \sum_{y\in\T_N^d} \Big\{  \frac{\partial_t \Phi (\rho_N(t,\tfrac{y}{N}))}{\Phi (\rho_N(t,\tfrac{y}{N}))} \eta(y) - \partial_t \log Z \big(\Phi (\rho_N(t,\tfrac{y}{N})) \big) \Big\}.\]
Since 
$E_{\nu_{N,t}} [\partial_t \log \psi_{N,t}] = E_{\nu_{\rho_*}^N} [\partial_t \psi_{N,t}] =0$ and $\rho_N (t,\frac{y}{N}) = E_{\nu_{N,t}} [\eta(y)] $,
we rewrite the last line as
\[N^{-d} \partial_t \log \psi_{N,t} = \frac{1}{N^d} \sum_{y\in\T_N^d}  \frac{\partial_t \Phi (\rho_N(t,\tfrac{y}{N}))}{\Phi (\rho_N(t,\tfrac{y}{N}))} \big(\eta(y)  - \rho_N(t,\tfrac{y}{N})\big).\]
Note that
\[\partial_t \Phi (\rho_N(t,u)) = N^{-\alpha}\Phi^\prime ( \rho_N (t,u)) \partial_t \rho (t,u) = \frac{\Phi^\prime (\rho_N(t,u))}{2N^\alpha} \sum_{i,j = 1}^d \sigma_{i,j} \Phi^\prime (\rho_*) \partial^2_{u_i,u_j} \rho (t,u).\]
Together with \eqref{lPsi1} and \eqref{entropy0}, we have 
\begin{multline}\label{entropy1}
\frac{dh_N(t)}{dt} \leq \int  \frac{1}{N^{d+\alpha}} \sum_{y\in\T_N^d} a_{N,t} (y) \big[g(\eta(y)) - \Phi (\rho_N(t,\tfrac{y}{N})) - \Phi^\prime (\rho_*) \big(\eta(y)  - \rho_N(t,\tfrac{y}{N})\big) \big] d \mu_{N,t}\\
	+ \int  \frac{1}{N^{d+2\alpha}} \sum_{y\in\T_N^d} b_{N,t} (y)  \big[g(\eta(y)) - \Phi (\rho_N(t,\tfrac{y}{N}))\big] d\mu_{N,t}+ o(N^{-2\alpha}).
\end{multline}
Above,
\begin{align*}
a_{N,t} (y) &= \frac{1}{2} \sum_{i,j=1}^d \frac{\sigma_{i,j}  \Phi^\prime (\rho_N(t,\tfrac{y}{N})) \partial^2_{u_i,u_j} \rho (t,\tfrac{y}{N})}{\Phi (\rho_N(t,\tfrac{y}{N}))} ,\\
b_{N,t} (y) &= \frac{1}{2} \sum_{i,j=1}^d  \frac{ \sigma_{i,j} \Phi^{\prime\prime} (\rho_N(t,\tfrac{y}{N}))\partial_{u_i} \rho (t,\tfrac{y}{N})  \partial_{u_j} \rho (t,\tfrac{y}{N})}{\Phi (\rho_N(t,\tfrac{y}{N}))} .
\end{align*}

\subsection{Proof of Theorem \ref{thm}.}\label{subsec:proof} In this subsection, we deal with the terms on the right hand side of \eqref{entropy1} respectively. We first deal with the second one since  it is simpler.

\begin{lemma}\label{lem:error0}
There exists a constant $C$ independent of $N$ such that
\[\int  \frac{1}{N^{d+2\alpha}} \sum_{y\in\T_N^d} b_{N,t} (y)  \big[g(\eta(y)) - \Phi (\rho_N(t,\tfrac{y}{N}))\big] d\mu_{N,t} \leq C N^{-3\alpha}.\]
\end{lemma}

\begin{proof}
Since $|\Phi (\rho_N(t,\tfrac{y}{N}))- \Phi (\rho_*)| \leq C N^{-\alpha}$, we only need to prove 
\begin{equation*}
	\int  \frac{1}{N^{d+2\alpha}} \sum_{y\in\T_N^d} b_{N,t} (y)  \big[g(\eta(y)) - \Phi (\rho_*)\big] d\mu_{N,t} \leq CN^{-3\alpha}.
\end{equation*}
By entropy inequality (\emph{cf.}\;\cite[Appendix 1.8]{klscaling} for example), the left hand side above is bounded by
\[	\frac{H(\mu_{N,t}|\nu_{\rho_*}^N)}{N^{d+\alpha}} + \frac{1}{N^{d+\alpha}} \log E_{\nu_{\rho_*}^N} \Big[ \exp \Big\{ \frac{1}{N^{\alpha}} \sum_{y\in\T_N^d} b_{N,t} (y)  \big[g(\eta(y)) - \Phi (\rho_*)\big] \Big\}\Big].\]
By Lemma \ref{lem:entropy}, $H(\mu_{N,t}|\nu_{\rho_*}^N) \leq CN^{d-2\alpha}$, hence the first term above is bounded by $CN^{-3\alpha}$. Since $\nu_{\rho_*}^N$ is product measure, the second term in the last expression equals
\begin{equation}\label{a1}
\frac{1}{N^{d+\alpha}} \sum_{y\in\T_N^d}  \log E_{\nu_{\rho_*}^N} \Big[ \exp \Big\{ \frac{1}{N^{\alpha}} b_{N,t} (y)  \big[g(\eta(y)) - \Phi (\rho_*)\big] \Big\}\Big].
\end{equation}

For a random variable $X$ such that $E[X] = 0$ and $E[e^{\lambda X}] < \infty$ for all $\lambda  \in \R$, we claim that for any $\lambda_0 > 0$, there exists a constant $C=C(\lambda_0)$ such that for all $0 < \lambda < \lambda_0$,
\begin{equation}\label{f1}
\log E [e^{\lambda X}] \leq C E[X^4]^{1/2} \lambda^2. 
\end{equation}
Indeed, since $e^x \leq 1 + x + (x^2/2) e^{|x|}$ and $\log (1+x) \leq x$, for any $0 < \lambda < \lambda_0$,
\[
\log E [e^{\lambda X}]  \leq \frac{\lambda^2}{2} E \big[ X^2 e^{\lambda |X|}\big] \leq \frac{\lambda^2}{2} E \big[ X^4]^{1/2} E \big[ e^{2 \lambda_0 |X|}\big]^{1/2}  =: C (\lambda_0)E[X^4]^{1/2} \lambda^2.
\]

Using the above claim, we bound \eqref{a1} by $CN^{-3\alpha}$ for large $N$. This completes the proof.
\end{proof}

Using the same argument as in Lemma \ref{lem:error0}, we could replace $\Phi^\prime (\rho_*)$ by $\Phi^\prime (\rho_N (t,\tfrac{y}{N}))$ in the first line in \eqref{entropy1} plus  an error of order $\Ocal (N^{-3\alpha})$. Up to now, we have shown
\begin{multline}\label{entropy2}
		\frac{dh_N(t)}{dt} \leq \int  \frac{1}{N^{d+\alpha}} \sum_{y\in\T_N^d} a_{N,t} (y) \big[g(\eta(y)) -\Phi (\rho_N(t,\tfrac{y}{N})) \\
		- \Phi^\prime (\rho_N(t,\tfrac{y}{N})) \big(\eta(y)  - \rho_N(t,\tfrac{y}{N})\big) \big] d \mu_{N,t}  
		+o(N^{-2\alpha}).
\end{multline}

\medspace

For a positive integer $\ell = \ell(N)$ and for any sequence $\{\phi(x)\}_{x \in \Z^d}$, define
\[\bar{\phi}^\ell (x) = \frac{1}{(2\ell+1)^d} \sum_{|y-x| \leq \ell} \phi (y).\]
In the following, we shall take
\begin{itemize}
	\item either $\ell = c N^{\frac{\alpha+d/2}{d+1}}$ for some sufficiently small  constant $c > 0$ when applying the spectral gap estimate (\emph{cf.} Lemma \ref{lem:spectralgap}),
		\item or $\ell = N^{\frac{\alpha+1}{d+1}}$ when applying the logarithmic Sobolev inequality (\emph{cf.} Lemma \ref{lem:lsi}).
\end{itemize}

\medspace

\begin{lemma}There exists a constant $C$ independent of $N$ such that
\begin{multline}\label{e1}
	\int  \frac{1}{N^{d+\alpha}} \sum_{y\in\T_N^d} \big[ a_{N,t} (y) - \bar{a}_{N,t}^\ell (y) \big]  \big[g(\eta(y)) - \Phi (\rho_N(t,\tfrac{y}{N})) \\- \Phi^\prime (\rho_N(t,\tfrac{y}{N})) \big(\eta(y)  - \rho_N(t,\tfrac{y}{N})\big) \big] d \mu_{N,t} 
	\leq \frac{C \ell^2}{N^{2+2 \alpha}}.
\end{multline}
In particular, with the above choice of $\ell$,  the above term is of order $o(N^{-2\alpha})$.
\end{lemma}

\begin{remark}
	By Taylor's expansion, the above formula is bounded by $C \ell^2/ N^{2+\alpha}$.  However,  the bound is of order $o(N^{-2\alpha})$ only for $\alpha < (d+2)/(d+3)$ when $\ell = c N^{\frac{\alpha+d/2}{d+1}}$, and only for $\alpha < 2d / (d+3)$ when $\ell = N^{\frac{\alpha+1}{d+1}}$, which is not optimal. 
\end{remark}

\begin{proof}
	Since the proof is exactly the same as  in Lemma \ref{lem:error0}, we only sketch it. Note that  $\big| a_{N,t} (y) - \bar{a}_{N,t}^\ell (y) \big| \leq C \ell^2 / N^2$. Let
	\[A_{N,t}^\ell (\eta) = \frac{N^2}{\ell^2} \big[ a_{N,t} (0) - \bar{a}_{N,t}^\ell (0) \big]  \big[g(\eta(0)) - \Phi (\rho_*) - \Phi^\prime (\rho_*) \big(\eta(0)  - \rho_* \big) \big].\]
	We pay a price of order $\Ocal (\ell^2 N^{-2-2\alpha})$ by replacing $\rho_N (t,\frac{y}{N})$ with $\rho_*$ in \eqref{e1}. By entropy inequality, Lemma \ref{lem:entropy} and inequality \eqref{f1},
	\begin{multline*}
		\int \frac{\ell^2}{N^{d+\alpha+2}} \sum_{y \in \T_N^d} \tau_y A_{N,t}^\ell (\eta) d \mu_{N,t} \leq \frac{H(\mu_{N,t} | \nu_{\rho_*}^N) }{N^{d+2} \ell^{-2}} \\
		+ \frac{1}{N^{d+2} \ell^{-2}} \sum_{y \in \T_N^d} \log \int \exp \Big\{ \frac{1}{N^\alpha} \tau_y A^\ell_{N,t} (\eta)\Big\} d \nu_{\rho_*}^N \leq \frac{C \ell^2}{N^{2+2\alpha}}.
  	\end{multline*}
  This concludes the proof.
\end{proof}

 Using summation by parts formula,
\begin{multline}\label{b1}
\int  \frac{1}{N^{d+\alpha}} \sum_{y\in\T_N^d} \bar{a}_{N,t}^\ell (y)  \big[g(\eta(y)) - \Phi (\rho_N(t,\tfrac{y}{N}))- \Phi^\prime (\rho_N(t,\tfrac{y}{N})) \big(\eta(y)  - \rho_N(t,\tfrac{y}{N})\big) \big] d \mu_{N,t} \\
= \int  \frac{1}{N^{d+\alpha}} \sum_{y\in\T_N^d} a_{N,t} (y)  \big[\bar{g}^\ell (\eta(y)) - \Phi (\bar{\rho}^\ell_N(t,\tfrac{y}{N})) - \Phi^\prime (\bar{\rho}^\ell_N(t,\tfrac{y}{N})) \big(\bar{\eta}^\ell(y)  - \bar{\rho}^\ell_N(t,\tfrac{y}{N})\big) \big] d \mu_{N,t} \\
+  \sum_{j=2}^3 \mathcal{E}_{N,j},
\end{multline}
where
\begin{align*}
\mathcal{E}_{N,2} =&  \frac{1}{N^{d+\alpha}} \sum_{y\in\T_N^d} a_{N,t} (y) \Big[ \Phi (\bar{\rho}^\ell_N(t,\tfrac{y}{N})) - \bar{\Phi}^\ell (\rho_N(t,\tfrac{y}{N})) \Big],\\
\mathcal{E}_{N,3} =&  \int  \frac{1}{N^{d+\alpha}} \sum_{y\in\T_N^d} a_{N,t} (y) \Big[ \Phi^\prime (\bar{\rho}^\ell_N(t,\tfrac{y}{N})) \big(\bar{\eta}^\ell(y) - \bar{\rho}^\ell_N(t,\tfrac{y}{N})\big) \\
&- \frac{1}{(2\ell+1)^d} \sum_{|x-y| \leq \ell} \Phi^\prime (\rho_N(t,\tfrac{x}{N})) \big(\eta(x)  - \rho_N(t,\tfrac{x}{N})\big)  \Big] d \mu_{N,t}.
\end{align*}
Since for any $|x-y| \leq \ell$, $|\rho_N(t,\tfrac{x}{N}) - \rho_N (t,\tfrac{y}{N}) |\leq C \ell N^{-1-\alpha}$,  it is easy to see
\[\big| \mathcal{E}_{N,j} \big| \leq C \ell N^{-1-2\alpha} = o(N^{-2\alpha}), \quad j = 2,3.\]

In order to deal with the first term on the right hand side of \eqref{b1}, we need the following one-block estimate, whose proof is  postponed to Section \ref{sec:oneblock}.

\begin{lemma}[One-block estimate]\label{lem:oneblock}
With the choice of $\ell(N)$ as above, for any continuous function $F: \R_+ \times \T^d \rightarrow \R$, 
\[\int_0^t E_{\mu_{N,s}} \Big[ \frac{1}{N^{d+\alpha}} \sum_{y\in\T_N^d} F \big( s, \tfrac{y}{N}\big) \tau_y V_g^\ell (\eta) \Big] ds = o(N^{-2\alpha}),\]
where
\[V_g^\ell (\eta) :=  \bar{g}^\ell (\eta(0)) - \Phi \big( \bar{\eta}^{\ell} (0)\big) .\]
\end{lemma}

\medspace

By Lemma \ref{lem:oneblock} and \eqref{b1}, we have shown that
\begin{equation}\label{entropy3}
	h_N(t) \leq h_N (0) + \int^t_0 E_{\mu_{N,s}} \Big[ \frac{1}{N^{d+\alpha}} \sum_{y\in\T_N^d} a_{N,s} (y) W^\ell_y (s,\eta)\Big] ds + o(N^{-2\alpha}),
\end{equation}
where
\[W^\ell_y (t, \eta) := \Phi (\bar{\eta}^\ell (y)) -  \Phi (\bar{\rho}^\ell_N(t,\tfrac{y}{N})) - \Phi^\prime (\bar{\rho}^\ell_N(t,\tfrac{y}{N})) \big(\bar{\eta}^\ell(y)  - \bar{\rho}^\ell_N(t,\tfrac{y}{N})\big). \]

In order to deal with the integral in \eqref{entropy3}, we first state a lemma and refer the readers to  \cite[P.193]{toth2002between} for its proof.

\begin{lemma}[{\cite[Lemma 2]{toth2002between}}]\label{lem:ineqn1}
	Let $\zeta_i,\,i\geq 1,$ be independent random variables with mean zero. For any $\lambda > 0$, assume 
	\[\Lambda_i (\lambda) = \log E [e^{\lambda \zeta_i}] < \infty.\]  
	Assume further that there exist constants $\lambda_0 > 0$ and $C_0 > 0$ such that  $\Lambda_i (\lambda)  \leq C_0 \lambda^2$ for all $0 < \lambda < \lambda_0$ and for all $i \geq 1$. Let $G: \R \rightarrow \R_+$ be smooth and satisfy $G(u) \leq C_1 (|u| \wedge u^2)$ for some constant $C_1 > 0$. Denote $S_\ell = \sum_{i=1}^\ell \zeta_i$. Then,  there exist constants $\gamma_0 > 0$ and  $C_2 < \infty$ such that for any $\ell > \gamma_0^{-1}$, for any $0 < \gamma < \gamma_0$,
	\[\log E \big[ \exp \big\{ \gamma \ell G(S_\ell/ \ell)\big\}\big] < C_2.\]
\end{lemma}

\medspace

Now we continue to treat \eqref{entropy3}. By entropy inequality, for any $\gamma > 0$, 
\[E_{\mu_{N,s}} \Big[ \frac{1}{N^{d+\alpha}} \sum_{y\in\T_N^d} a_{N,s} (y)  W^\ell_y (s, \eta)  \Big] \leq \frac{h_N(s)}{\gamma} + \frac{1}{\gamma N^d} \log E_{\nu_{N,s}} \Big[ \exp \Big\{ \frac{\gamma}{N^\alpha} \sum_{y\in\T_N^d} a_{N,s} (y)  W^\ell_y (s, \eta)\Big\}\Big]. \]
Since $W^\ell_y (s, \eta)$ and $ W^\ell_{y^\prime} (s, \eta) $ are independent under $\nu_{N,s}$ if $|y-y^\prime| > 2 \ell$, by H{\"o}lder's inequality, the second term in the last line is bounded by
\[ \frac{1}{\gamma N^d (2\ell+1)^d} \sum_{y\in\T_N^d} \log E_{\nu_{N,s}} \Big[ \exp \Big\{ \frac{\gamma (2\ell+1)^d}{N^\alpha} a_{N,s} (y) W^\ell_y (s, \eta) \Big\}\Big].\]
Take $\gamma = \gamma_0 N^\alpha$ for some fixed but small enough $\gamma_0$ and take
\[G(u) = \Phi (u+\bar{\rho}_N^\ell (t,y)) -  \Phi (\bar{\rho}_N^\ell (t,\tfrac{y}{N})) -  \Phi^\prime (\bar{\rho}_N^\ell (t,\tfrac{y}{N})) u, \quad u =\bar{\eta}^\ell (y) - \bar{\rho}_N^\ell (t,\tfrac{y}{N})\]
in Lemma \ref{lem:ineqn1}, then the above term is bounded by $C \ell^{-d} N^{-\alpha}$ for some constant $C = C(\gamma_0)$.  Therefore,
\begin{equation}\label{entropy4}
h_N(t) \leq h_N (0) + \frac{C}{N^\alpha} \int_0^t h_N (s) ds +  \frac{C}{\ell^d N^\alpha} + o(N^{-2\alpha}).
\end{equation}
With of choices of $\ell$ above, 
\[\frac{1}{\ell^d N^\alpha} = \Ocal(N^{-\frac{(4d+2)\alpha+d^2}{2(d+1)}}) \quad \text{or} \quad  \frac{1}{\ell^d N^\alpha} =  \Ocal (N^{-\frac{(2d+1)\alpha + d}{d+1}}).\]
In both cases, the third term in \eqref{entropy4} is of order $o(N^{-2\alpha})$ if $d \geq 2$ and $\alpha \in (0,1)$.  We conclude the proof of Theorem \ref{thm} by using Gr{\"o}nwall's inequality.

\section{One-block Estimate}\label{sec:oneblock}

In this section, we prove Lemma \ref{lem:oneblock}. In order to cut off large densities, for each $M > 0$, let 
\[V_{g,M}^\ell (\eta) =V_g^\ell ( \eta) \mathbf{1}_{\{ \bar{\eta}^{\ell} (0) \leq M\}}. \]
Then, the expression in Lemma \ref{lem:oneblock} equals
\begin{multline}\label{oneblock0}
\int_0^t E_{\mu_{N,s}} \Big[ \frac{1}{N^{d+\alpha}} \sum_{y\in\T_N^d} F \big( s, \tfrac{y}{N}\big) \tau_y \big(V_g^\ell (\eta) \mathbf{1}_{\{ \bar{\eta}^{\ell} (0) > M\}} \big) \Big] ds\\
+ \int_0^t E_{\mu_{N,s}} \Big[ \frac{1}{N^{d+\alpha}} \sum_{y\in\T_N^d} F \big( s, \tfrac{y}{N}\big) \tau_y V_{g,M}^\ell (\eta) \Big] ds.
\end{multline}

\subsection{Cut off large densities.}  In this subsection, we bound the first term  in \eqref{oneblock0}.

\begin{lemma}\label{cutoff} There exists some constant $C=C(M,F,t)$ such that for $N$ large enough,
\[\int_0^t E_{\mu_{N,s}} \Big[ \frac{1}{N^{d+\alpha}} \sum_{y\in\T_N^d} F \big( s, \tfrac{y}{N}\big) \tau_y \big(V_g^\ell (\eta) \mathbf{1}_{\{ \bar{\eta}^{\ell} (0) > M\}} \big) \Big] ds \leq C  \Big( \frac{\log N}{N^{3\alpha}} + \frac{\log N}{\ell^d N^\alpha} \exp \{ - C  \ell^d\} \Big).\]
In particular, with the choices of $\ell=\ell(N)$ given in Subsection \ref{subsec:proof}, the above term has order $o(N^{-2\alpha})$.
\end{lemma}

\begin{proof}
By Assumption \ref{assump:g} $(i)$,  $\big| V_g^\ell (\eta) \big| \leq C(a_0) \bar{\eta}^\ell (0)$. Therefore, the  term on the right hand side in the lemma  is bounded by
\[C \|F\|_\infty \int_0^t E_{\mu_{N,s}}  \Big[  \frac{1}{N^{d+\alpha}} \sum_{y\in\T_N^d} \bar{\eta}^\ell (y)  \mathbf{1}_{\{ \bar{\eta}^{\ell} (y) > M\}}  \Big] ds.\] 
By entropy inequality \cite[Appendix 1.8]{klscaling}, for any $\gamma > 0$, the integral above is bounded by
\begin{equation}\label{cutoff1}
\int_0^t \frac{H (\mu_{N,s} | \nu_{\rho_*}^N)}{\gamma} ds + \frac{t}{\gamma} \log E_{\nu_{\rho_*}^N} \Big[ \exp \Big\{   \frac{\gamma}{N^{d+\alpha}} \sum_{y\in\T_N^d} \bar{\eta}^\ell (y)  \mathbf{1}_{\{ \bar{\eta}^{\ell} (y) > M\}} \Big\}\Big]
\end{equation}
By Lemma \ref{lem:entropy}, for any $0 \leq s \leq t$,
\[H (\mu_{N,s} | \nu_{\rho_*}^N) \leq H (\mu_{N,0} | \nu_{\rho_*}^N)  \leq C N^{d-2\alpha}.\]
For the second term in \eqref{cutoff1}, since $\bar{\eta}^\ell (y)$ and $\bar{\eta}^\ell (z)$ are independent if $|y-z| > 2 \ell$,  by H{\" o}lder's inequality,
\begin{multline*}
\log E_{\nu_{\rho_*}^N} \Big[ \exp \Big\{   \frac{\gamma}{N^{d+\alpha}} \sum_{y\in\T_N^d} \bar{\eta}^\ell (y)  \mathbf{1}_{\{ \bar{\eta}^{\ell} (y) > M\}} \Big\} \Big] \\
\leq \frac{N^d}{(2\ell+1)^d}   \log E_{\nu_{\rho_*}^N} \Big[ \exp \Big\{   \frac{\gamma (2\ell+1)^d}{N^{d+\alpha}} \bar{\eta}^\ell (0)  \mathbf{1}_{\{ \bar{\eta}^{\ell} (0) > M\}} \Big\} \Big]
\end{multline*}
By using the basic inequality $\log (1+x) \leq x$, we bound the last line by
\begin{multline*}
\frac{N^d}{(2\ell+1)^d}  \log E_{\nu_{\rho_*}^N} \Big[ 1+ \exp \Big\{   \frac{\gamma (2\ell+1)^d}{N^{d+\alpha}} \bar{\eta}^\ell (0) \Big\}   \mathbf{1}_{\{ \bar{\eta}^{\ell} (0) > M\}} \Big]\\
\leq \frac{N^d}{(2\ell+1)^d}   E_{\nu_{\rho_*}^N} \Big[ \exp \Big\{   \frac{\gamma (2\ell+1)^d}{N^{d+\alpha}} \bar{\eta}^\ell (0)  \Big\}   \mathbf{1}_{\{ \bar{\eta}^{\ell} (0) > M\}} \Big].
\end{multline*}
By \CS,  the expectation in the last line is bounded by
\begin{equation}\label{cutoff2}
	E_{\nu_{\rho_*}^N} \Big[ \exp \Big\{   \frac{2 \gamma (2\ell+1)^d}{N^{d+\alpha}} \bar{\eta}^\ell (0)  \Big\} \Big]^{1/2}  P_{\nu_{\rho_*}^N} \Big( \bar{\eta}^{\ell} (0) > M \Big)^{1/2}.
\end{equation}
For $\lambda \geq  0$ and $\theta \in \R$,  define
\[\Lambda (\lambda) := \log E_{\nu_{\rho_*}^N} \big[ e^{\lambda \eta(0)}\big], \quad I(\theta) := \sup_{\lambda} \{ \lambda \theta - \Lambda (\lambda)\}.\]
Recall under Assumption \ref{assump:g}, $\Lambda (\lambda) < \infty$ for any $\lambda \geq 0$. By standard large deviation estimates,  we may bound \eqref{cutoff2} by
\[	\exp \Big\{ C(2\ell+1)^d \big[ \Lambda (2\gamma / N^{d+\alpha}) - I(M) \big]\Big\}.\]
To sum up, we bound \eqref{cutoff1} by 
\begin{equation}\label{cutoff3}
	\frac{C \|F\|_\infty tN^{d}}{\gamma N^{2\alpha}} + \frac{ C \|F\|_\infty t N^d}{\gamma (2\ell+1)^d} \exp \Big\{ C(2\ell+1)^d \big[ \Lambda (2\gamma / N^{d+\alpha}) - I(M) \big]\Big\}.
\end{equation}
Now, take $\gamma = N^{d+\alpha} / \log N$. Since $\lim_{\lambda \rightarrow 0} \Lambda (\lambda) = 0$, for $N$ large enough, \[\Lambda (2\gamma / N^{d+\alpha}) = \Lambda(2/\log N) < I(M)/2.\] Hence, \eqref{cutoff3} is bounded by
\[C \|F\|_\infty t \Big( \frac{\log N}{N^{3\alpha}} + \frac{\log N}{\ell^d N^\alpha} \exp \{ - C I (M) \ell^d\} \Big).\]
This concludes the proof of the lemma.
\end{proof}

\subsection{Estimates by spectral gap inequality.} In this subsection, we bound the second term in \eqref{oneblock0} by using the spectral gap estimates for the zero range process.

\begin{lemma}\label{lem:spectralgap}
Let $F: \R_+ \times \T^d \rightarrow \R$ be continuous. There exists $\varepsilon_0 > 0$ such that for any $\ell$ satisfying $M \ell^{d+1} < \varepsilon_0 N^{d/2 + \alpha}$, there exists $C = C(M,F,t)$ such that
\begin{equation}\label{ob}
\int_0^t E_{\mu_{N,s}} \Big[ \frac{1}{N^{d+\alpha}} \sum_{y\in\T_N^d} F \big( s,\tfrac{y}{N}\big) \tau_y V_{g,M}^\ell  (\eta) \Big] ds \leq C \Big( \frac{1}{N^\alpha \ell^d} + \frac{ \ell}{N^{2\alpha+d/2}}\Big).
\end{equation}
In particular, by taking 
\[\ell = \Big(\frac{\varepsilon_0}{2M}\Big)^{1/(d+1)} N^{\frac{\alpha+d/2}{d+1}},\]
the left hand side in \eqref{ob} is bounded by $C N^{-\frac{(4d+2)\alpha + d^2}{2(d+1)}} = o(N^{-2\alpha})$.
\end{lemma}

\begin{proof}
By entropy inequality, for any $\gamma > 0$,  
\begin{multline}\label{oneblock1}
	\int_0^t E_{\mu_{N,s}} \Big[ \frac{1}{N^{d+\alpha}} \sum_{y\in\T_N^d} F \big( s,\tfrac{y}{N}\big) \tau_y V_{g,M}^\ell  \Big] ds \leq \frac{H(\mu_{N,0} | \nu_{\rho_*}^N)}{\gamma N^d} \\
	+ \frac{1}{\gamma N^d} \log \mathbb{E}_{\nu_{\rho_*}^N} \Big[ \exp \Big\{ \int_0^t \frac{\gamma \|F\|_\infty}{N^\alpha}   \sum_{y\in\T_N^d} \big|  \tau_y V_{g,M}^\ell (\eta_s) \big| ds \Big\}\Big].
\end{multline}
By Lemma \ref{lem:entropy}, $H(\mu_{N,0} | \nu_{\rho_*}^N) \leq C N^{d-2\alpha}$.  Using the basic inequality 
\[e^{|x|} \leq e^x + e^{-x}, \forall x, \quad \text{and} \quad \log (a+b) \leq \log 2 + \max \{\log a, \log b\}, \forall a, b > 0,\]
we could remove the absolute value inside the exponential in the second term above. By Feynman-Kac formula \cite[Lemma A1.7.2]{klscaling}, the second term in \eqref{oneblock1} is bounded by $t \Gamma_N/(\gamma N^d)$,
where $\Gamma_N$ is the largest eigenvalue of the operator $N^2 \mathscr{L}_N^s +  \frac{\gamma \|F\|_\infty}{N^\alpha}   \sum_{y\in\T_N^d}  \tau_y V_{g,M}^\ell (\eta)$ with $\gen^s := (\gen + \gen^*)/2$ being the symmetric part of $\gen$ in $L^2 (\nu_{\rho_*}^N)$. Moreover, $t \Gamma_N/(\gamma N^d)$ may be written as the following variational form
\begin{equation}\label{oneblock2}
t \, \sup_{f:\nu_{\rho_*}^N-density} \Big\{ \frac{\|F\|_\infty}{N^{d+\alpha}} E_{\nu_{\rho_*}^N} \Big[   \sum_{y\in\T_N^d}  \tau_y V_{g,M}^\ell (\eta) f(\eta)  \Big]  - \frac{N^2}{\gamma N^d} D_N (f;\nu_{\rho_*}^N)\Big\}.
\end{equation}
Above,  recall $D_N (f;\nu_{\rho_*}^N)$ is the Dirichlet form of the density $f$ with respect to the measure $\nu_{\rho_*}^N$,
\[D_N (f;\nu_{\rho_*}^N) := \<\sqrt{f} (-\gen)\sqrt{f}\>_{\nu_{\rho_*}^N}  = \<\sqrt{f} (-\mathscr{L}_N^s)\sqrt{f}\>_{\nu_{\rho_*}^N}. \]
 For any $\nu_{\rho_*}^N$--density $f$, denote
\[\bar{f} = \frac{1}{N^d} \sum_{y\in\T_N^d} \tau_y f.\]
By translation invariance and convexity of the Dirichlet form,
\[D_N (f;\nu_{\rho_*}^N) \geq D_N (\bar{f};\nu_{\rho_*}^N).\]
This permits us to bound \eqref{oneblock2} by
\begin{equation}\label{oneblock3}
t \, \sup_{f} \Big\{  \frac{\|F\|_\infty}{N^{\alpha}} E_{\nu_{\rho_*}^N} \Big[  V_{g,M}^\ell (\eta) f(\eta)  \Big]  - \frac{N^2}{\gamma N^d} D_N (f;\nu_{\rho_*}^N)\Big\}.
\end{equation}
Above, the supremum is over all translation invariant $\nu_{\rho_*}^N$--densities $f$. 

Denote by $\Lambda_\ell^d := \{-\ell,-\ell+1, \ldots, \ell\}^d$ the $d$-dimensional cube of length $2\ell+1$ and denote $\Omega_\ell^d := (\Lambda_\ell^d)^\N$. Let $\nu_{\rho_*}^\ell$ be the product measure on $\Omega_\ell^d$ with particle density $\rho_*$. Denote $f_\ell := E_{\nu_{\rho_*}^N}  [f | \mathcal{F}_\ell]$ with $\mathcal{F}_\ell$ being the $\sigma$-algebra generated by $(\eta(x), x\in \Lambda_\ell^d)$.   Let 
\begin{align*}
D_\ell (f;\nu_{\rho_*}^N) &= \sum_{x,y \in \Lambda_\ell^d} D_{x,y} (f;\nu_{\rho_*}^N), \\
 D_\ell (f_\ell;\nu_{\rho_*}^\ell) &= \<\sqrt{f}(-\mathscr{L}_\ell^s) \sqrt{f}\>_{\nu_{\rho_*}^\ell}\\
 &=  \frac{1}{2} \sum_{x,y \in \Lambda_\ell^d} \int_{\Omega_\ell^d} g(\eta(x)) s(y-x) \big[ \sqrt{f(\eta^{x,y})} - \sqrt{f(\eta)} \big]^2 d \nu_{\rho_*}^\ell,
\end{align*}
where the symmetric generator $\mathscr{L}_\ell^s$ acts on functions $f:\Omega_\ell^d \rightarrow \R$ as
\[\mathscr{L}_\ell^s f (\eta) = \sum_{x,y\in \Lambda_\ell^d} g(\eta(x)) s(y-x) \big[ f(\eta^{x,y}) - f(\eta) \big].\]
Since $f$ is translation invariant, by convexity of the Dirichlet form,
\[D_N (f;\nu_{\rho_*}^N) \geq\frac{N^d}{(2\ell+1)^d} D_\ell (f;\nu_{\rho_*}^N) \geq  \frac{N^d}{(2\ell+1)^d} D_\ell (f_\ell;\nu_{\rho_*}^\ell).\]
Therefore, \eqref{oneblock3} is bounded by
\begin{equation}\label{oneblock4}
	t \, \sup_{f} \Big\{  \frac{\|F\|_\infty}{N^{\alpha}} E_{\nu_{\rho_*}^\ell} \Big[  V_{g,M}^\ell (\eta) f_\ell (\eta)  \Big]  - \frac{N^2}{\gamma (2\ell+1)^d} D_\ell (f_\ell;\nu_{\rho_*}^\ell)\Big\}.
\end{equation}

To decompose the above term on hyperplanes with fixed number of particles, we first introduce the following notations.   For $j \geq 0$, let
\[ \Omega_{\ell,j}^d = \big\{ \eta \in \Omega_\ell^d: \sum_{x \in \Lambda_\ell^d} \eta (x) = j \big\}, \]
and let $\nu_{\ell,j}$ be the conditional measure of  $\nu_{\rho_*}^\ell$ on $\Omega_{\ell,j}^d$, \emph{i.e.}\;$\nu_{\ell,j} (\cdot) = \nu_{\rho_*}^\ell (\cdot | \Omega_{\ell,j}^d)$.  Note that $\nu_{\ell,j}$ is independent of the density $\rho_*$. Define
\[m_{\ell,j} (f) = E_{\nu_{\rho_*}^\ell} [f_\ell \mathbf{1}_{\{\bar{\eta}^\ell (0)= j / (2\ell+1)^d\}}], \quad f_{\ell,j} = \frac{f_\ell}{E_{\nu_{\ell,j}} [f_\ell]}. \]
Then,
\[E_{\nu_{\rho_*}^\ell} \Big[  V_{g,M}^\ell (\eta) f_\ell (\eta)  \Big] = \sum_{j=0}^{M (2\ell+1)^d} m_{\ell,j} (f) E_{\nu_{\ell,j}} [V_{g,M}^\ell f_{\ell,j}].\]
Moreover, 
\[D_\ell (f_\ell;\nu_{\rho_*}^\ell) = \sum_{j=0}^{\infty} m_{\ell,j} (f) D_\ell (f_{\ell,j};\nu_{\ell,j}) \geq \sum_{j=0}^{M (2\ell+1)^d} m_{\ell,j} (f) D_\ell (f_{\ell,j};\nu_{\ell,j}).\]
Since $f_\ell$ is density with respect to $\nu_{\rho_*}^\ell$,
\[\sum_{j=0}^{M (2\ell+1)^d} m_{\ell,j} (f)  \leq E_{\nu^\ell_{\rho_*}} [f_\ell] = 1.\]
Then, we may bound \eqref{oneblock4} by
\begin{equation}\label{oneblock5}
t \sup_{f} \sup_{0 \leq j \leq (2\ell+1)^d M}  \Big\{ \frac{\|F\|_\infty}{N^\alpha} E_{\nu_{\ell,j}} [V_{g,M}^\ell f_{\ell,j}] - \frac{N^d}{\gamma (2\ell+1)^d} D_\ell (f_{\ell,j};\nu_{\ell,j}) \Big\}
\end{equation}

Note that $V_{g,M}^\ell$ is not mean zero with respect to the measure $\nu_{\ell,j}$. In order to make it centralized, define
\begin{align*}
U_{g,M}^{\ell,j} :=& \;\bar{g}^\ell(\eta(0)) \mathbf{1}_{\{\bar{\eta}^\ell (0) \leq M\}}- E_{\nu_{\ell,j}} \big[ \bar{g}^\ell(\eta(0)) \mathbf{1}_{\{\bar{\eta}^\ell (0) \leq M\}} \big]\\
=&  \;\bar{g}^\ell(\eta(0)) \mathbf{1}_{\{\bar{\eta}^\ell (0) \leq M\}}- E_{\nu_{\ell,j}} \big[ g(\eta(0))\big].
\end{align*}
The last identity holds because $j \leq M (2\ell+1)^d$. We may rewrite $V_{g,M}^\ell$ as 
\[V_{g,M}^\ell = U_{g,M}^{\ell,j} + E_{\nu_{\ell,j}} \big[ g(\eta(0))\big] - \Phi (\bar{\eta}^\ell (0)) \mathbf{1}_{\{\bar{\eta}^\ell (0) \leq M\}}.\] 
Let $\lambda_g^{\ell,j} (-\mathscr{L}_\ell^s)$ be the largest eigenvalue associated to the generator $-\mathscr{L}_\ell^s + \gamma (2\ell+1)^d \|F\|_\infty U_{g,M}^{\ell,j} / N^{d+\alpha}$, 
Then, \eqref{oneblock5} is bounded by
\begin{equation}\label{oneblock6}
\frac{t \|F\|_\infty}{N^\alpha}  \sup_{0 \leq j \leq (2\ell+1)^d M } \big\{ E_{\nu_{\ell,j}} [g(\eta(0))] - \Phi \big(j/(2\ell+1)^d\big)\big\}+ \frac{t N^d}{\gamma (2\ell+1)^d} \;  \sup_{0 \leq j \leq (2\ell+1)^d M}  \lambda_g^{\ell,j} (-\mathscr{L}_\ell^s).
\end{equation}
By the standard equivalence of ensembles \cite[Corollary A2.1.7]{klscaling}, 
\[ \sup_{0 \leq j \leq (2\ell+1)^d M } \big\{ E_{\nu_{\ell,j}} [g(\eta(0))] - \Phi \big(j/(2\ell+1)^d\big)\big\} \leq C \ell^{-d}.\]
Hence, the first term  in \eqref{oneblock6} is bounded by $C N^{-\alpha} \ell^{-d}$.   In order to bound the second term in \eqref{oneblock6}, we need the following lemma to bound  the largest eigenvalue of the small perturbation of the reversible generator $-\mathscr{L}_\ell^s$.

\begin{lemma}[{\cite[Theorem A3.1.1]{klscaling}}] Let $\mathscr{L}$ be a reversible irreducible generator on some countable state space $\Omega$ with respect to some probability measure $\nu$ on $\Omega$. Let $U$ be a mean zero bounded function with respect to $\nu$. For $\varepsilon > 0$, denote by $\lambda_\varepsilon$ the upper bound of the spectrum of $\mathscr{L}+\varepsilon U$. Then,
	\[0 \leq \lambda_\varepsilon \leq \frac{\varepsilon^2 \kappa}{1-2\|U\|_\infty \varepsilon \kappa} {\rm Var} (U;\nu),\]
	where $\kappa^{-1}$ is the spectral gap of the generator $\mathscr{L}$ in $L^2 (\nu)$,
	\[\kappa^{-1} = \inf_{f:E_{\nu}[f] = 0} \frac{\<f(-\mathscr{L})f\>_{\nu}}{{\rm Var} (f;\nu)}\]
\end{lemma}

We continue to bound the second term in \eqref{oneblock6}. Note that $E_{\nu_{\ell,j}} [U_{g,M}^{\ell,j}] = 0$ and $|U_{g,M}^{\ell,j}| \leq CM$.  By the above lemma, the second term  in \eqref{oneblock6} is bounded by
\[\sup_{0 \leq j \leq (2\ell+1)^d M } \Big\{  \frac{C(F,t) N^d}{\gamma \ell^d} \frac{\gamma^2 \ell^{2d} N^{-2d-2\alpha}  \kappa}{1-C(F) M \gamma \ell^d N^{-d-\alpha} \kappa} {\rm Var} \big( U_{g,M}^{\ell,j}; \nu_{\ell,j} \big) \Big\} ,\]
where $\kappa = \kappa (\ell,j)$ is the reciprocal of the spectral gap of the generator  $\mathscr{L}_\ell^s$,
\[\kappa^{-1} = \inf_{f} \frac{\<f(-\mathscr{L}_\ell^s)f\>_{\nu_{\ell,j}}}{{\rm Var} (f;\nu_{\ell,j})}.\]
Above, the infimum is over all functions $f$ such that $E_{\nu_{\ell,j}} [f] = 0$.  In \cite{landim1996spectral}, Landim \emph{et al.} prove the following lower bound for the spectral gap of the zero range process under Assumption \ref{assump:g},  \[\kappa (\ell,j)^{-1} \geq C \ell^{-2}\] 
uniformly in $j$. By \cite[Lemma 4]{toth2002between},  
\[\sup_{0 \leq j \leq (2\ell+1)^dM} {\rm Var} \big( U_{g,M}^{\ell,j}; \nu_{\ell,j} \big) \leq C \ell^{-d}.\]
Therefore, the second term in \eqref{oneblock6} is bounded by $C (M,F,t)\gamma \ell^{2} N^{-d-2\alpha}$ as long as $\gamma M \ell^{d+2} < \varepsilon_0 N^{d+\alpha}$ for some fixed but small $\varepsilon_0$.  Adding up the above estimates \eqref{oneblock1} and \eqref{oneblock6}, we bound the left hand side in \eqref{ob} by
\[C (M,F,t) \Big( \frac{1}{N^\alpha \ell^d} + \frac{\gamma \ell^2}{N^{d+2\alpha}} + \frac{1}{\gamma N^{2\alpha}} \Big).\]
Take $\gamma = N^{d/2} \ell^{-1}$, we get the desired bound
\[C (M,F,t) \Big( \frac{1}{N^\alpha \ell^d} + \frac{ \ell}{N^{2\alpha+d/2}}\Big)\]
as long as $M \ell^{d+1} < \varepsilon_0 N^{d/2 + \alpha}$, which concludes the proof.
\end{proof}

\subsection{Estimates by logarithmic Sobolev inequality.}   In this subsection, we present a second proof to bound the second term in \eqref{oneblock0} by using the logarithmic Sobolev inequality for the zero range process.  

\begin{lemma}\label{lem:lsi}
	There exists $C = C(M,F,t)$ such that
	\begin{equation}\label{lsi}
		\int_0^t E_{\mu_{N,s}} \Big[ \frac{1}{N^{d+\alpha}} \sum_{y\in\T_N^d} F \big( s,\tfrac{y}{N}\big) \tau_y V_{g,M}^\ell  \Big] ds \leq C \Big( \frac{1}{N^\alpha \ell^d} + \frac{ \ell^{d+2}}{N^{3\alpha+2}}\Big).
	\end{equation}
In particular, by taking $\ell = N^{\frac{\alpha+1}{d+1}}$, the left hand side in \eqref{lsi} is bounded by $C N^{-\frac{(2d+1)\alpha+d}{d+1}}$, which has order $o(N^{-2\alpha})$.
\end{lemma}

\begin{remark}
	Compared the estimate in the above lemma with that in Lemma \ref{lem:spectralgap}, we get a better bound by using logarithmic Sobolev inequality only in dimension one. Indeed,
	\[N^{-\frac{(2d+1)\alpha+d}{d+1}} < N^{-\frac{(4d+2)\alpha + d^2}{2(d+1)}}  \quad \text{if and only if} \quad d < 2.\]
\end{remark}

\begin{proof}
	Denote
	\[f_s = f_{N,s} = \frac{d \mu_{N,s}}{d \nu_{\rho_*}^N}, \quad \bar{f}_s = \frac{1}{N^d} \sum_{y \in \T_N^d} \tau_y f_s, \quad  \bar{f}_{s,\ell} = E_{\nu_{\rho_*}^N} [ \bar{f}_s | \mathcal{F}_\ell]. \]
	Then, we may bound the left hand side in \eqref{lsi} by
	\[ \int_0^t E_{\nu_{\rho_*}^N} \Big[ \frac{\|F\|_\infty	}{N^{\alpha}}   V_{g,M}^\ell \bar{f}_s  \Big] ds =  \int_0^t E_{\nu_{\rho_*}^\ell} \Big[ \frac{\|F\|_\infty	}{N^{\alpha}}   V_{g,M}^\ell \bar{f}_{s,\ell}  \Big] ds\]
	As in the proof of Lemma \ref{lem:spectralgap}, we denote
	\[m_{\ell,j} (\bar{f}_{s,\ell}) = E_{\nu_{\rho_*}^\ell} [\bar{f}_{s,\ell} \mathbf{1}_{\{\bar{\eta}^\ell (0)= j / (2\ell+1)^d\}}], \quad\bar{f}_{s,\ell,j}= \frac{\bar{f}_{s,\ell}}{E_{\nu_{\ell,j}} [\bar{f}_{s,\ell}]}. \]
	Then, we rewrite the last formula as 
	\[\sum_{j=0}^{M(2\ell+1)^d}   \int_0^t m_{\ell,j} (\bar{f}_{s,\ell})  E_{\nu_{\ell,j}} \Big[ \frac{\|F\|_\infty	}{N^{\alpha}}   V_{g,M}^\ell \bar{f}_{s,\ell,j}  \Big] ds. \]
	By entropy inequality, for any $\gamma > 0$, the above formula is bounded by 
	\begin{multline}\label{c1}
	\sum_{j=0}^{M(2\ell+1)^d}  \frac{1}{\gamma} \int_0^t  m_{\ell,j} (\bar{f}_{s,\ell}) E_{\nu_{\ell,j}} \big[ \bar{f}_{s,\ell,j} \log \bar{f}_{s,\ell,j} \big] ds \\
	+ \sum_{j=0}^{M(2\ell+1)^d}  \frac{1}{\gamma}  \int_0^t  m_{\ell,j} (\bar{f}_{s,\ell}) \log E_{\nu_{\ell,j}} \Big[ \exp \Big\{ \frac{\gamma \|F\|_\infty}{N^\alpha} V_{g,M}^\ell \Big\}\Big] ds.
	\end{multline}

The logarithmic Sobolev inequality for the zero range process reads (\emph{cf.} \cite{pra2005logarithmic})
\[E_{\nu_{\ell,j}} \big[ \bar{f}_{s,\ell,j} \log \bar{f}_{s,\ell,j} \big] \leq C \ell^2 D_\ell \big(  \bar{f}_{s,\ell,j}; \nu_{\ell,j}\big)\]
uniformly in $j$. By convexity and translation invariance of the Dirichlet form, 
\[\sum_{j=0}^{M(2\ell+1)^d}  m_{\ell,j} (\bar{f}_{s,\ell}) D_\ell \big(  \bar{f}_{s,\ell,j}; \nu_{\ell,j}\big) \leq D_\ell \big(  \bar{f}_{s,\ell}; \nu_{\rho_*}^\ell \big) \leq  D_\ell \big(  \bar{f}_{s}; \nu_{\rho_*}^N \big) \leq \frac{(2\ell+1)^d}{N^d} D_N  \big(  \bar{f}_{s}; \nu_{\rho_*}^N \big).\]
Since by Lemma \ref{lem:entropy},
\[\int_0^t D_N  \big(  \bar{f}_{s}; \nu_{\rho_*}^N \big) ds \leq C N^{d-2\alpha -2},\]
we bound the first term in \eqref{c1} by $C \ell^{d+2} / (\gamma N^{2\alpha+2})$.

Now we bound the second term in \eqref{c1}.  Note that $\big| V_{g,M}^\ell \big| \leq C M$. Using the basic inequality
\[e^{x} \leq 1 + x +\tfrac{x^2}{2} e^{|x|}, \quad \log (1+x) \leq x,\]
we bound the second term in \eqref{c1} by
\[\frac{C (M,F,t) }{\gamma} \sup_{0 \leq j \leq M (2\ell+1)^d} \Big\{ \frac{\gamma}{N^\alpha}E_{\nu_{\ell,j}} [V^\ell_{g,M}] +  \frac{\gamma^2}{N^{2\alpha}}E_{\nu_{\ell,j}} [(V^\ell_{g,M})^2] \Big\}, \quad \forall \gamma \leq N^{\alpha}.\]
Using the standard equivalence of ensembles (\emph{cf.}\;\cite[Corollary A2.1.7]{klscaling}), for any $j \leq M(2\ell+1)^d$,
\[E_{\nu_{\ell,j}} [V^\ell_{g,M}] = E_{\nu_{\ell,j}} [g(\eta(0))] - \Phi (j/(2\ell+1)^d) \leq C \ell^{-d}.\]
By \CS \;and \cite[Lemma 4]{toth2002between}, 
\[E_{\nu_{\ell,j}} [(V^\ell_{g,M})^2] \leq 2 \big(E_{\nu_{\ell,j}} [g(\eta(0))] - \Phi (j/(2\ell+1)^d)\big)^2 + 2 {\rm Var} \big( \bar{g}^\ell (0); \nu_{\ell,j}\big) \leq C \ell^{-d}.\]
This permits us to bound the second term in \eqref{c1} by $C (M,F,t) \Big(N^{-\alpha} \ell^{-d} + \gamma N^{-2\alpha} \ell^{-d}\Big)$.

In conclusion, for any $\gamma \leq N^{\alpha}$, we bound  the left hand side in \eqref{lsi} by
\[C (M,F,t) \Big( \frac{1}{N^\alpha \ell^d} + \frac{\ell^{d+2}}{\gamma N^{2\alpha+2}} + \frac{\gamma }{N^{2\alpha} \ell^d}\Big).\]
We finish the proof by taking $\gamma = N^\alpha$.
\end{proof}

\bibliographystyle{plain}
\bibliography{zhaoreference.bib}

\vspace{1em}
\noindent{\large Linjie \textsc{Zhao}}

\vspace{0.5em}
\noindent School of Mathematics and Statistics, Huazhong University of Science \& Technology, Wuhan, China.\\
{\tt linjie\_zhao@hust.edu.cn}
\end{document}